\documentclass{amsart}
\usepackage{amsthm,amsmath,amsfonts,amssymb,array,enumerate,verbatim,enumitem}
\usepackage[all]{xy}
\usepackage{url}

\newcommand{\ZZ}{\mathbb{Z}}

\newcommand{\QQ}{\mathbb{Q}}
\newcommand{\cO}{\mathcal{O}}

\newcommand{\fm}{\mathfrak{m}}

\newcommand{\vL}{\  \big\rvert \ }
\DeclareMathOperator{\Hom}{Hom}

\DeclareMathOperator{\Ind}{Ind}

\DeclareMathOperator{\End}{End}

\DeclareMathOperator{\Spec}{Spec}

\DeclareMathOperator{\Gal}{Gal}
\DeclareMathOperator{\Fr}{Fr}
\DeclareMathOperator{\tr}{tr}

\DeclareMathOperator{\Char}{Char}
\DeclareMathOperator{\Sp}{Sp}
\DeclareMathOperator{\Sw}{Sw}

\DeclareMathOperator{\rk}{rk}

\newtheorem{thm}{Theorem}[section]
\newtheorem{lemma}[thm]{Lemma}
\newtheorem{prop}[thm]{Proposition}
\newtheorem{cor}[thm]{Corollary}
\newtheorem{defn}[thm]{Definition}

\newtheorem{rmk}[thm]{Remark}

\numberwithin{equation}{section}

\title{Deligne--Langlands gamma factors in families}
\date{\today}
\author{David Helm}
\address{David Helm\\
Imperial College, London}
\email{d.helm@imperial.ac.uk}

\author{Gilbert Moss}
\address{Gilbert Moss\\
Institut de Math\'{e}matiques de Jussieu}
\email{gil.moss@imj-prg.fr}

\begin{document}
\begin{abstract}
Let F be a $p$-adic field, $W_F$ its absolute Weil group, and let $k$ be an algebraically closed field of prime characteristic $\ell$ different from $p$. Attached to any $\ell$-adic representation of $W_F$ are local epsilon- and $L$-factors. There are natural notions of families of $\ell$-adic  representations of $W_F$, such as the theory of Galois deformations or, more generally, families over arbitrary Noetherian $W(k)$-algebras. However, the epsilon and $L$-factors do not interpolate well in such families. In this paper it is shown that the gamma factor, which is the product of the epsilon factor with a ratio of $L$-factors, interpolates over such families.
\end{abstract}
\maketitle
\section{Introduction}
\subsection{Motivation}
Let $F$ be a $p$-adic field whose residue field has order $q$, let $G_F$ be its absolute Galois group with respect to some fixed algebraic closure, and let $W_F \subset G_F$ be its Weil group.  For any prime $\ell\neq p$, and any $\ell$-adic representation $\rho$ of $W_F$ there is attached a rational function $L(\rho,X):=\det\left(I-\Fr X\vL \rho^{I_F}\right)^{-1}$, where $\rho^{I_F}$ denotes the inertial invariants and $\Fr$ is a geometric Frobenius element. In \cite[Th\'{e}or\`{e}me 4.1]{deligne72}, Deligne uses global methods to prove the existence of a local constant $\epsilon(\rho,\psi,dx)$ compatible with induction. One can also form the gamma factor,
$$\gamma(\rho,X,\psi):=\epsilon(\rho,\psi,dx)X^{a(\rho)}\frac{L(\rho^{\vee},\frac{1}{qX})}{L(\rho,X)},$$ where $a(\rho)$ is the Artin conductor. The gamma factors of certain twists uniquely determine $\rho$ up to semisimplification (e.g. \cite{hen_converse,jl}).

The deformation theory of Galois representations provides a natural notion of families of Galois representations, which are representable by spectra of complete local Noetherian $W(k)$-algebras where $k$ is an algebraically closed field of characteristic $\ell$. In section ~\ref{arbitrary}, we build on recent work of the first author~(\cite{curtis}) in developing a well-behaved notion of ``families of continuous representations of $W_F$'' over arbitrary Noetherian $W(k)$-algebras. 

However, $\epsilon(\rho,\psi, dx)$ and $L(\rho,X)$ are not compatible with reduction mod-$\ell$, and so do not interpolate well in such families. Our main result is showing that $\gamma(\rho,X,\psi)$ interpolates over such families.

\subsection{Statement of the main result}
Given $R$ a commutative ring, let $S$ be the multiplicative subset of $R[X,X^{-1}]$ consisting of Laurent polynomials whose first and last coefficients are units. The fraction ring $S^{-1}(R[X,X^{-1}])$ is a generalization of the ring of rational functions to the setting where $R$ is not necessarily a domain. The interpolated gamma factor forms an element of this ring:
\begin{thm}
\label{mainthm}
To any tuple $(R,\rho,\psi)$, where $R$ is a Noetherian $W(k)$-algebra, $\rho$ is an $\ell$-adically continuous representation $\rho:W_F\to GL_n(R)$ in the sense of Section ~\ref{arbitrary} below, and $\psi:F\rightarrow W(k)^{\times}$ is a locally constant character, we may associate an element $\gamma_R(\rho,X,\psi)\in S^{-1}(R[X,X^{-1}])$ such that the following properties are satisfied:
\begin{enumerate}
\item Let $R'$ be a Noetherian $W(k)$-algebra and let $f:R\rightarrow R'$ be a morphism.  Then, extending $f$ to a homomorphism $f:S^{-1}R[X,X^{-1}]\rightarrow S^{-1}R'[X,X^{-1}]$, we have $$f(\gamma_R(\rho,X,\psi)) = \gamma_{R'}(\rho\otimes_RR',X,\psi).$$
\item Suppose in addition that $R$ is the fraction field of a complete local domain of characteristic zero with residue field $k$ and $\rho$ is an $\ell$-adic representation over $R$ (see Definition \ref{elladicrepdef}). Then $\gamma_R(\rho,X,\psi)$ is equal to the classical Deligne--Langlands gamma factor $\gamma(\rho,X,\psi)$ of Equation \ref{gammadefnincharzero}.
\end{enumerate}
The association $(R,\rho,\psi) \mapsto \gamma_R(\rho,X,\psi)$ is uniquely determined by
the above properties, and additionally satisfies:
\begin{enumerate}[resume]
\item Given an exact sequence $0\rightarrow \rho'\rightarrow \rho\rightarrow \rho''\rightarrow 0$, we have $\gamma_R(\rho,X,\psi) = \gamma_R(\rho',X,\psi)\gamma_R(\rho'',X,\psi)$.
\item If $E/F$ is a finite separable extension and $r_E$ is a virtual smooth representation of $W_E$ with virtual degree zero, then $\gamma_R(\Ind_{E/F}r_E, X,\psi)=\gamma_R(r_E,\psi\circ\tr_{E/F})$.
\end{enumerate}
\end{thm} 
\subsection{Relationship to previous work}
This result is the natural analogue, on the Galois side of the local Langlands correspondence, of prior work of the second author on interpolating gamma factors over families of admissible representations
of $GL_n(F)$.  There is the standard gamma factor for pairs of such representations \cite{jps2}, and a notion of families of such representations over arbitrary Noetherian $W(k)$-algebras, namely the ``co-Whittaker families'' of \cite{h_whitt}.  The second author's thesis \cite{moss1,moss2} shows that the gamma factors of \cite{jps1,jps2} interpolate nicely across such families (whereas the $L$- and $\epsilon$-factors do not) and satisfy a converse theorem in families.

The local Langlands correspondence is the unique sequence $\{\pi_n\}$ of canonical bijections from the set of $n$-dimensional Frobenius-semisimple Weil-Deligne representations of $W_F$ to the set of irreducible admissible smooth representations of $GL_n(F)$ (up to isomorphism). The correspondences $\pi_n$ are canonical in the sense that $\gamma(\rho\otimes \rho',X,\psi)=\gamma(\pi_n(\rho)\times \pi_m(\rho'),X,\psi)$, where the right-hand side is the standard gamma factor of \cite{jps2}, and $\rho$ and $\rho'$ have dimensions
$n$ and $m$ respectively. 

It is expected that the local Langlands correspondence is compatible with variation in families. This idea is due to Emerton and the first author in \cite{eh}.  Conjecture 1.3.1 of \cite{eh} associates a co-Whittaker module to any family of Galois representations in a way that interpolates the local Langlands correspondence across the family. Therefore, the good interpolation properties of the gamma factor for admissible representations established in \cite{moss1,moss2} \emph{should} have analogues on the Galois side of the local Langlands correspondence. Our main result, Theorem \ref{mainthm}, confirms this expectation.

The question of interpolating the local constant $\epsilon(\rho,\psi,dx)$ has been addressed in \cite{deligne72}, \cite{yasuda}, and \cite{kestutis}. Deligne introduces in \cite{deligne72} the modified epsilon constant $\epsilon_0(\rho,\psi,dx)=\epsilon(\rho,\psi,dx)\det\left(-\Fr\vL \rho^{I_F}\right)$ and shows that it is compatible with reduction mod $\ell$, thus interpolating it to discrete valuation rings of residue characteristic $\neq p$. By following Laumon in the positive characteristic setting, Yasuda constructs $\epsilon_0(\rho,\psi,dx)$ using the geometric Fourier-Deligne transform, extending $\epsilon_0(\rho,\psi,dx)$ to the setting of Noetherian local rings with residue field $k$ such that $(R^{\times})^p=R^{\times}$. In \cite{kestutis}, \v{C}esnavi\v{c}ius uses formal commutative algebra to interpolate $\epsilon_0(\rho,\psi,dx)$ over normal integral $\ZZ[\frac{1}{p}]$-schemes.

Like $\epsilon_0$, the gamma factor is compatible with specialization, but is arguably a richer object than $\epsilon_0$. Indeed, $\gamma(\rho,X,\psi)$ is a formal power series (rather than a constant) that encodes additional information about the Frobenius eigenvalues. Converse theorems (e.g. \cite{hen_converse,jl}) imply that collections of Deligne--Langlands gamma factors of twists uniqely determine $\rho$ up to semisimplification. Nonetheless, there is little mention of the Deligne--Langlands gamma factor in the literature (however, an interesting interpolation property of a ratio of adjoint gamma values has been observed in \cite{gross_reeder}).

\subsection{Applications}

By \cite[Thm 7.9]{h_whitt}, to prove the local Langlands correspondence in families it suffices to construct a map from the integral Bernstein center (which serves as the base of a
``universal co-Whittaker family'') to the universal base ring $R^{\nu}$ of $\ell$-adic $W_F$-representations (see Theorem \ref{universal}) that is ``compatible with local Langlands,'' in a certain sense. Since the initial version of this note was written, the authors have found a proof that such a map exists, thereby proving Conjecture 1.3.1 of \cite{eh}. The proof appears in a pair of recent preprints, \cite{converse} and \cite{curtis}. Theorem \ref{mainthm} of the present paper plays a key role in the proof.

For the sake of motivation, if we take for granted the local Langlands correspondence in families (as stated in \cite{converse}) we can describe two properties our gamma factors are expected to satisfy. First, the converse theorem \cite[Thm 1.1]{moss2} would imply that, when $R$ is reduced, finite-type, and flat (e.g. $R^{\nu}$), the collection of $\gamma_R(\rho\otimes\rho',X,\psi)$, for suitable twists $\rho'$, determines $\rho$ up to semisimplification. 

Second, let $\fm$ be a maximal ideal of $R^{\nu}$. There would be a map of $W(k)$-modules, call it $\theta_{\nu,n}:R^{\nu}\to W(k)$, which is the analogue of the map $\theta_{e,n}$ in \cite[\S 3]{converse}. Then Corollary 6.5 of \cite{converse} would imply that the images in $id\otimes\theta_{\nu,n}$ of the coefficients of $\gamma(\rho^{\nu}\otimes\rho',X,\psi)$, for various $\rho'$, provides a set of topological generators for the subring $R_{\overline{\rho^{\nu}_{\fm}}}^{\square,\text{inv}}$ of frame invariant elements of the framed deformation ring $R_{\overline{\rho^{\nu}_{\fm}}}^{\square}$.


\subsection{Outline}
Section \ref{arbitrary} extends the notion of $\ell$-adically continuous representations of $W_F$ to arbitrary Noetherian $W(k)$-algebras, relying on results in \cite{curtis}. 

Section \ref{swan} develops the Swan conductor in the general setting of Section \ref{arbitrary} by interpolating a construction of Katz. 

Section \ref{fieldsection} gathers relevant aspects of the classical theory of Dwork--Langlands--Deligne local constants in the setting of fields. 

Section \ref{epsilon0} extends the results of \cite{yasuda} on $\epsilon_0$, reducing the problem of interpolating $\gamma(\rho,X,\psi)$ to the problem of interpolating a ratio of $L$-functions.

Section \ref{provingrestofgamma} approaches the ratio of $L$-functions in the tamely ramified setting using new and quite explicit techniques, c.f. Theorem \ref{restofgamma}. 

Section \ref{gammaforrings} combines Theorem \ref{restofgamma} and Proposition \ref{epsilon_0reducedelltorsionfree} to deduce Theorem \ref{mainthm}.

\subsection{Acknowledgements}

The first author was supported by EPSRC grant\\ EP/M029719/1. The second author is grateful to K\c{e}stutis \u{C}esnavi\u{c}ius and Keenan Kidwell for several helpful comments. Both authors are grateful to Guy Henniart for his enthusiasm and comments, and for pointing out the simpler proof of Lemma \ref{permutationdiagonal}.

\section{Families over arbitrary Noetherian base rings}
\label{arbitrary}
Recall that a $W(k)$-algebra $R$ is said to be {\em $\ell$-adically separated} if the intersection
of the ideals $\ell^i R$ is equal to the zero ideal. We will denote by $K^i$ the kernel of the map $GL_n(R) \to GL_n(R/\ell^i R)$.

\begin{defn}
\label{ladiccontinuity} 
\begin{enumerate}
\item
Let $R$ be an $\ell$-adically separated Noetherian $W(k)$-algebra, and let $\rho: W_F \rightarrow GL_n(R)$ be a homomorphism.  We say $\rho$ is $\ell$-adically continuous if, for all $i$, the preimage of the subgroup $K^i$ of $GL_n(R)$ under $\rho$ is open in $W_F$.
\item More generally, if $R$ is an arbitrary Noetherian $W(k)$-algebra, we say  
$\rho$ is $\ell$-adically continuous if there is an affine open cover $U_i=\Spec R_i$ of $\Spec R$ such that, for each $i$, there exists an $\ell$-adically separated Noetherian $W(k)$-algebra $R_i'$ with a map $R_i'\to R$ and an $\ell$-adically continuous representation $\rho_i:W_F\to GL_n(R_i')$ such that $\rho\otimes_RR_i\cong \rho_i\otimes_{R_i'}R_i$. (Colloquially: it arises, Zariski-locally, by base-change from objects in Definition \ref{ladiccontinuity}(1)).
\end{enumerate}
\end{defn}

\begin{lemma}
\label{finitequotient}
Let $R$ be a Noetherian $W(k)$-algebra and $\rho: W_F \rightarrow GL_n(R)$ an $\ell$-adically continuous representation, and let $I_F$ be the inertia subgroup. If $H$ is a closed subgroup of $I_F$ of pro-order prime to $\ell$, then $\rho(H)$ is finite.
\end{lemma}
\begin{proof}
It suffices to prove this when $R$ is $\ell$-adically separated.  Let $H'$ be the intersection of $H$ with the kernel of the composition $W_F\to GL_n(R)\to GL_n(R/\ell R)$.  For each $i$, the image of $K^i$ in the map $GL_n(R) \to GL_n(R/\ell^{i+1} R)$ is an abelian group of exponent $\ell$. Since $\rho(H')$ is in $K^1\cap \rho(H)$, its image in each map $GL_n(R) \to GL_n(R/ \ell^i R)$ must be trivial. In particular, $\rho(H')$ must lie in $K^i\subset I_n + \ell^i M_n(R)$ for all $i$.  Thus $\rho(H')$ is trivial.
\end{proof}

\begin{cor} \label{cor:sum}
Let $R$ be a Noetherian $W(k)$-algebra and $\rho: W_F \rightarrow GL_n(R)$ an $\ell$-adically continuous
representation.  Let $H$ be a closed subgroup of $I_F$ of pro-order prime to $\ell$.  Then there is a direct sum decomposition:
$$\rho= \bigoplus \rho_{\tau}$$
where $\tau$ runs over the irreducible $k$-representations of $H$, and for each $\tau$, the restriction
of $\rho_{\tau}$ to $H$ is isomorphic to ${\tilde \tau} \otimes_{W(k)} M_{\tau}$, where $\tilde \tau$ is the unique lift of $\tau$ to $W(k)$ and $M_{\tau}$ is a locally free $R$-module.
\end{cor}
\begin{proof}
By the previous lemma there is an $H'$ open normal in $H$ such that $\rho$ is trivial on $H'$.  Set
$\rho_{\tau} = 0$ if $\tau$ is not trivial on $H'$, and $\rho_{\tau} = e_{\tau} \rho$ otherwise, where
$e_{\tau}$ is the idempotent of $W(k)[H/H']$ corresponding to $\tau$.  Since $\tilde \tau$ is a projective
$W(k)[H]$-module that generates the block corresponding to $e_{\tau}$, the natural map:
$$\Hom_{W(k)[H]}(\tilde \tau, \rho_{\tau}) \otimes \tilde \tau \rightarrow \rho_{\tau}$$
is an isomorphism.  Let $M_{\tau} = \Hom_{W(k)[H]}(\tilde \tau, \rho_{\tau}).$
Since $\tilde \tau$ is free as a $W(k)$-module it follows that
$M_{\tau}$ is a direct summand of $\rho$, which is a free $R$-module of finite rank,
and is therefore projective and hence locally free.
\end{proof}

Let $I_F^{(\ell)}$ denote the prime-to-$\ell$ part of $I_F$, that is, the kernel of any, and hence every, surjective homomorphism $I_F \to\ZZ_{\ell}$. Note that Corollary \ref{cor:sum} implies that if $R$ has connected spectrum, and $\rho:W_F \rightarrow GL_n(R)$ is $\ell$-adically continuous then the restriction of $\rho$ to a closed subgroup $H$ of prime-to-$\ell$ inertia is constant, that is, arises by base change from $W(k)$. It also implies $\rho|_H=\rho^{ss}|_H$, where $\rho^{ss}$ denotes the semisimplification of $\rho$.

Let $\nu$ be an $n$-dimensional $k$-representation of $I_F^{(\ell)}$ that extends to $W_F$.  Then $\nu$ has a unique lift $\tilde{\nu}$ to $W(k)$. We will say that an $\ell$-adically continuous representation $\rho: W_F \rightarrow GL_n(R)$ \emph{has type $\nu$} if, Zariski locally on $\Spec R$, the restriction of $\rho$ to $I_F^{(\ell)}$ is isomorphic to $\tilde{\nu}\otimes_{W(k)}R$. Note that if $R$ is connected, then $\rho$ has a type.

\begin{defn} If $\rho$ is an $\ell$-adically continuous representation of $W_F$ over $R$, a {\em pseudo-framing}
of $R$ is a choice of $R$-basis for $M_{\tau}$, for each $\tau$.
\end{defn}

Since the $M_{\tau}$ are in general only locally free over $R$, an arbitrary $\ell$-adically continuous
representation $\rho$ of $W_F$ over $R$ will only admit a pseudo-framing Zariski locally on $R$.  The space of pseudo-framings of $\rho$ is a torsor for an algebraic group $G^{\nu}$, which is a product of groups $GL_{n_{\tau}}$,
where $n_{\tau}$ is the rank of $M_{\tau}$.

\begin{thm}[\cite{curtis}, \S 9] \label{universal} There exists a reduced, finite type, $\ell$-torsion free, $\ell$-adically separated Noetherian $W(k)$-algebra $R^{\nu}$, and an $\ell$-adically continuous representation $\rho_{\nu}: W_F \rightarrow GL_n(R_{\nu})$, with a canonical pseudo-framing,
such that if $R$ is any Noetherian $W(k)$-algebra,
and $\rho: W_F \rightarrow GL_n(R)$ any $\ell$-adically continuous representation having type $\nu$ and
a pseudo-framing, there is a unique map
$R^{\nu} \rightarrow R$ such that $\rho$ is isomorphic to $\rho_{\nu} \otimes_{R^{\nu}} R$ and the pseudo-framing
on $\rho$ is the one induced by base change from the pseudo-framing on $\rho_{\nu}$.

Moreover, if $x: R^{\nu} \rightarrow k$ is any map, and $\rho_x: W_F \rightarrow GL_n(k)$ the corresponding representation, then the completion of $R^{\nu}$ at the kernel of $x$, taken together with the base change of $\rho_{\nu}$ to this completion, gives the universal pseudo-framed deformation of $\rho_x$.
\end{thm}

The universal property of $R^{\nu}$ shows that the scheme $\Spec R^{\nu}$ admits an action of $G^{\nu}$ by ``change of pseudo-framing''.  One can form the quotient stack $(\Spec R^{\nu})/G^{\nu}$, which is a moduli stack parameterizing representations $\rho$ of $W_F$ on locally free $R$-modules that, {\em Zariski locally on $\Spec R$}, are $\ell$-adically continuous representations in the sense of Definition \ref{ladiccontinuity}.  We will not adopt this perspective here, however.

\begin{cor} 
If $R$ is a Noetherian $W(k)$-algebra, and $\rho: W_F \rightarrow GL_n(R)$ is an $\ell$-adically continuous
representation, then, Zariski locally on $\Spec R$, there exists a $\nu$, and a map $R^{\nu} \rightarrow R$ of $W(k)$-algebras, such that $\rho$ is isomorphic to the base change of $\rho^{\nu}$ to $R$.
\end{cor}

\begin{cor}
\label{ladicmaxadic}
Let $R$ be a complete local Noetherian $W(k)$-algebra with residue field $k$, and maximal ideal $\fm$. Then $\rho$ is $\ell$-adically continuous in the sense of Definition \ref{ladiccontinuity} if and only if it is $\fm$-adically continuous.
\end{cor}
\begin{proof}
Since $R$ is $\ell$-adically separated, and the $\ell$-adic topology is coarser than the $\fm$-adic topology, the ``only if'' direction is immediate. For the other direction, let $\overline{\rho}$ denote the representation $\rho \otimes_R k$; since $\overline{\rho}$ is $\ell$-adically continuous, it admits a type $\nu$.  Choose a pseudo-framing on
$\overline{\rho}$; this gives rise to a map $R^{\nu} \rightarrow k$, which has kernel a maximal ideal $\mathfrak p$,
such that $\overline{\rho}$ is isomorphic to $\rho^{\nu} \otimes_{R^{\nu}} k$.

Lift the pseudo-framing on $\overline{\rho}$ to a pseudo-framing on $\rho$.  Then we can regard $\rho$ as a
pseudo-framed deformation of $\overline{\rho}$.  On the other hand, the completion of $R^{\nu}$ at $\mathfrak p$
is the universal pseudo-framed deformation of $\overline{\rho}$.  Thus there is a map: 
$$(R^{\nu})_{\mathfrak p} \rightarrow R$$
such that $\rho$ arises by base change from $R^{\nu}$ via the composition:
$$R^{\nu} \rightarrow (R^{\nu})_{\mathfrak p} \rightarrow R.$$
In particular, $\rho$ is $\ell$-adically continuous.
\end{proof}

\section{the Swan conductor}
\label{swan}
In the case of a smooth representation over any field of characteristic different from $p$, the Swan conductor is given by an explicit formula in terms of the lower numbering filtration on $I_F$ (\cite[4.5]{deligne72}, \cite[\S 4]{ulmer}). The connection between the conductor of a representation and the upper numbering filtration was known to Howe in the 1970's: \cite[Thm 3.5]{hen_thesis}. In \cite{katz88}, Katz defines the Swan conductor for continuous representations over complete local Noetherian rings in terms of the upper numbering break decomposition. We extend Katz' definition to the setting of arbitrary Noetherian $W(k)$-algebras.

For any real number $v\geq -1$, let $G_F^v$ denote the upper numbering ramification subgroups of $G_F$ and define $G^v:=W_F\cap G_F^v$. Define $G^{v+}:=\bigcup_{w>v}G^w$. Then $G^0=I_F$ and the closure of $G^{0+}$ is the wild inertia subgroup $P_F$. If $R$ is a Noetherian $W(k)$-algebra and $\rho$ is an $R[W_F]$-module, we say that $\rho$ is \emph{pure of break $v$} if the fixed vectors $\rho^{G^v}$ are zero and $\rho|_{G^{v+}}$ is trivial. 

Let $\rho$ be $\ell$-adically continuous. Letting the subgroup $H$ in Lemma \ref{finitequotient} run over the closures of $G^{v+}$, we have (c.f. \cite[Lemma 1.4]{katz88}) that if $\rho$ is $\ell$-adically continuous, then for $v\in \QQ_{\geq 0}$ there exists a unique subrepresentation $\rho^v$ which is pure of break $v$, and such that $$\rho = \bigoplus_{v\in \QQ_{v\geq 0}}\rho^v,$$ with finitely many nonzero summands. Moreover the functor $\rho\mapsto \rho^v$ is exact and commutes with any base change $R\to R'$. When $R$ is \emph{local}, each $\rho^v$ is projective over $R$, hence free, and we define

\begin{defn}
\label{swandef}
\begin{enumerate}
\item Let $R$ be a Noetherian local $W(k)$-algebra and let $\rho:W_F\to R$ be an $\ell$-adically continuous representation with break decomposition $\bigoplus_{v\in \QQ_{\geq 0}} \rho^v$. Define the \emph{Swan conductor}
$$\Sw_R(\rho) = \sum_{v\in\QQ_{\geq 0}}v\cdot\rk(\rho^v).$$
\item Let $R$ be an arbitrary Noetherian $W(k)$-algebra with connected spectrum, let $\rho$ be an $\ell$-adically continuous representation, and let 
$$\rho = \bigoplus_{\tau}\rho_{\tau} \cong \bigoplus_{\tau} M_{\tau} \otimes {\tilde \tau}$$ 
be the decomposition from Corollary \ref{cor:sum} in terms of irreducible $k$-representations $\tau$ of $P_F$, and let $n_{\tau}$ be the rank of the locally free $R$-module $M_{\tau}$. Define the \emph{Swan conductor} $$\Sw_R(\rho) = \sum_{\tau}n_{\tau}\Sw_k(\tau).$$ 
\end{enumerate}
\end{defn}

In both cases we observe that the Swan conductor depends only on the restriction to $P_F$ (the $v=0$ term vanishes in the first definition).  
\begin{prop}
\label{elladicswan}
\begin{enumerate}
\item If $R$ is a Noetherian local $W(k)$-algebra, then Definitions \ref{swandef}(1) and \ref{swandef}(2) give nonnegative integer values and are equivalent.
\item Let $R$ be a Noetherian $W(k)$-algebra with connected spectrum. Then $\Sw_R(\rho_x)$ is constant as $x$ varies over $\Spec(R)$.
\end{enumerate}
\end{prop}
\begin{proof}
By \cite[Prop 1.9]{katz88}, when $R$ is a complete local ring with finite residue field, Definition \ref{swandef}(1) agrees with the lower-number definition given in \cite[6.2]{deligne72}, and defines an integer. By descent to the case of a finite residue field, it follows that $\Sw_k(\tau)=\Sw_{W(k)}(\tilde{\tau})=\Sw_{W(k)[1/\ell]}(\tilde{\tau}[1/\ell])$ is an integer. When $R$ is local, $\Sw_R$ from Definition \ref{swandef}(1) is compatible with specialization and additive in direct sums, so we have proved (1). If $\Spec(R)$ is connected, restriction of $\rho$ to $P_F$ is constant, proving (2).
\end{proof}
\begin{rmk}
As an aside, the proof of \cite[Prop 2.11]{kestutis} also proves the following statement. If $R$ is any Noetherian $\ZZ[\frac{1}{p}]$-algebra with connected spectrum, and $\rho$ is an $R[W_F]$-module that factors through a finite quotient of $I_F$, the Swan conductor is constant on $\Spec(R)$.
\end{rmk}
If $\Spec(R)$ is not connected, we define the Swan conductor as a tuple of integers, consisting of the conductors of the projections of $\rho$ onto each connected component. Except for Theorem \ref{universal} and Corollary \ref{ladicmaxadic}, the preceding discussion works verbatim with $W(k)$ replaced by $\ZZ_{\ell}$ (though we will not use this).

\section{Dwork--Langlands--Deligne gamma factors}
\label{fieldsection}
Given any $\ZZ[\frac{1}{p}]$-algebra $R$, we fix a Haar measure $dx$ on $F$ valued in $R$ such that $\int_{\cO_F}dx=1$ (\cite[6.1]{deligne72}). For a topological ring $R$, a representation $r:W_F\to GL_n(R)$ is \emph{smooth} if it is continuous with respect to the discrete topology on $R$. Let $R$ be a commutative local ring in which $p$ is invertible, $\chi$ a smooth character $F^{\times}\to R^{\times}$, and $\psi:F\to W(k)^{\times}$ a smooth additive character. Denote by $n(\psi)$ the smallest integer $m$ such that $\psi|_{\varpi^m\cO_F}$ is trivial. If $c$ is any element of $F$ with valuation $n(\psi)+\Sw(\chi) +1$, we define 
\begin{equation}
\label{epsilonforcharacters}
\epsilon(\chi,\psi) = \left\{
\begin{array}{ccl}
1 & \text{if} &\chi\text{ is unramified and } n(\psi)=0\\
\int_{c^{-1}\cO_F^{\times}}\chi^{-1}(x)\psi(x)dx &\text{if} & \chi\text{ is ramified }
\end{array}\right.
\end{equation}
and $\epsilon(\chi, \psi(ax)) = \chi(a)q^{-v(a)}\epsilon(\chi,\psi)$ in general.

\begin{thm}[\cite{deligne72}, Th\'{e}or\`{e}me 4.1]
\label{fieldsepsilon}
Let $\kappa$ be a field of characteristic zero, let $r:W_F\rightarrow GL_n(\kappa)$ be a smooth representation, let $\psi:W_F\rightarrow \kappa^{\times}$ be a smooth character. There is a unique element $\epsilon(r,\psi)$ in $\kappa^{\times}$, depending only the isomorphism class of $r$, satisfying the following conditions:
\begin{enumerate}
\item For any exact sequence $0\to r'\to r\to r''\to 0$, we have $\epsilon(r,\psi)=\epsilon(r',\psi)\epsilon(r'',\psi)$.
\item If $E/F$ is a finite separable extension, and $r_E$ is a virtual representation of $W_E$ with virtual degree zero, then $\epsilon(\Ind_{E/F}r_E,\psi)=\epsilon(r_E,\psi\circ \tr_{E/F})$.
\item If $\dim r=1$ then $\epsilon(r,\psi)$ coincides with Equation (\ref{epsilonforcharacters}).
\end{enumerate}
\end{thm}

The \emph{Artin conductor} $a(r)$ is defined in \cite[4.5]{deligne72} (or \cite{ulmer}). The monomial $\epsilon(r,X,\psi)$ is defined by twisting $r$ by the unramified character $\Fr\mapsto X:W_F\to \kappa[X,X^{-1}]^{\times}$, the transformation formula \cite[5.5.1]{deligne72} then writes it as $\epsilon(r,\psi)X^{a(r)+\dim(r)n(\psi)}$.

The epsilon factor $\epsilon(r,X,\psi)$ is not compatible with specialization. However, Deligne's modified epsilon factor is. 
\begin{equation}
\label{modifiedepsilondef}
\epsilon_0(r,\psi) = \epsilon(r,\psi)\det\left(-\Fr\vL{r^I}\right).
\end{equation}

\begin{thm}[\cite{deligne72} Th\'{e}or\`{e}me 6.5]
\label{modifiedepsilonfactor}
Let $\kappa$ be a field of characteristic different from $p$, let $r:W_F\rightarrow GL_n(\kappa)$ be a smooth representation, and let $\psi:W_F\rightarrow \kappa^{\times}$ be a smooth character. There exists a unique element $\epsilon_0(r,\psi)$ in $\kappa^{\times}$ satisfying the following conditions:
\begin{enumerate}
\item $\epsilon_0(r,\psi)$ is invariant under extension of the field $\kappa$.
\item For any exact sequence $0\to r'\to r\to r''\to 0$, we have $\epsilon_0(r,\psi)=\epsilon_0(r',\psi)\epsilon_0(r'',\psi)$.
\item If $E/F$ is a finite separable extension, and $r_E$ is a virtual representation of $W_E$ with virtual degree zero, then $\epsilon_0(\Ind_{E/F}r_E,\psi)=\epsilon_0(r_E,\psi\circ \tr_{E/F})$.
\item Let $\cO$ be a discrete valuation ring with residue characteristic different from $p$, fraction field $\kappa$, and residue field $\cO/\fm$. If $r$ and $\psi$ are defined over $\cO$, then $\epsilon_0(r,\psi)$ lies in $\cO^{\times}$ and satisfies $\epsilon_0(r,\psi)\equiv \epsilon_0(\overline{r},\overline{\psi})\mod \fm$ where $\overline{r}$ and $\overline{\psi}$ are the reductions of $r$ and $\psi \mod \fm$.
\item If $r$ is a character, then $\epsilon_0(r,\psi)$ coincides with Equation (\ref{modifiedepsilondef}).
\end{enumerate}
\end{thm}

We define the L-factor and gamma factor as
\begin{align*}
L(r,X):=\det\left(1-\Fr X\vL{r^{I_F}}\right)^{-1}\\ 
\gamma(r,X,\psi) := \epsilon(r,X,\psi)\frac{L(r^{\vee},\frac{1}{qX})}{L(r,X)},
\end{align*}
where $r^{\vee}$ denotes the contragredient.

We extend these definitions to the setting of Weil-Deligne representations, or equivalently, $\ell$-adic representations. First recall the necessary definitions.

If $\kappa$ is a field of characteristic different from $p$, a Weil-Deligne representation is a pair $(r,N)$ consisting of a smooth representation $r:W_F\to GL_n(\kappa)$ and a nilpotent endomorphism $N$ of $r$ such that $wNw^{-1}=|w|N$ where $|\cdot |$ denotes the power of $\Fr$ occurring in $w$. The Frobenius-semisimplification of a Weil-Deligne representation $(r,N)$ is the representation $(r^{ss},N)$ where $r^{ss}(w)$ denotes the semisimple part of $r(w)$. The dual Weil-Deligne representation is $(r^{\vee},N^{\vee})$ where $N^{\vee}\in\End_{\kappa}(r)$ is the operator such that $(N^{\vee}\phi)(v)=\phi(-Nv)$ for $\phi\in r^{\vee},v\in r$. 

\begin{defn}
\label{elladicrepdef}
Suppose $R$ is the fraction field of a complete local Noetherian domain $R_0$ of characteristic zero with residue field $k$ and maximal ideal $\fm$. We say $\rho:W_F\to GL_n(R)$ is an \emph{$\ell$-adic representation over $R$} to mean it is isomorphic to $\rho_0\otimes_{R_0}R$ for some $\rho_0:W_F\to GL_n(R_0)$ that is continuous with respect to the $\fm$-adic topology on $R_0$ (hence $\rho$ is $\ell$-adically continuous in the sense of Definition \ref{ladiccontinuity}). 
\end{defn}
There is an equivalence (in fact an isomorphism) of categories between $\ell$-adic representations over $R$ and Weil-Deligne representations over $R$ \cite[4.1.6 Prop]{eh}. Note that semisimple $\ell$-adic representations are smooth \cite[(4.2.3) Corollary]{tate_ntb}. 

Let $(r,N)$ be a Weil-Deligne representation over a field $\kappa$ of characteristic $0$, and let $r_N=\ker(N)$. If, in addition, $\kappa$ is the fraction field of a complete local Noetherian domain of characteristic zero with residue field $k$, let $\rho$ be the $\ell$-adic representation over $\kappa$ corresponding to $(r,N)$. Following \cite[8.12]{deligne72} and \cite[4.1.6]{tate_ntb} we set

\begin{align}
L((r,N),X) &= L(r_N,X)=L(\rho,X)\\
\epsilon((r,N),X,\psi)& = \epsilon(r,X,\psi)\det(-\Fr X\vL r^I/r_N^I)=\epsilon(\rho,X,\psi)\\
\epsilon_0((r,N),\psi) &=\epsilon_0(r,\psi)=\epsilon_0(\rho,\psi)\\
\label{gammadefnincharzero}\gamma((r,N),X,\psi)&=\epsilon((r,N),X,\psi)\frac{L((r^{\vee},N^{\vee}),\frac{1}{qX})}{L((r,N),X)} = \gamma(\rho,X,\psi)
\end{align}

\begin{lemma}
\label{extendingtoweildeligne}
Let $(r,N)$ be a Frobenius semisimple Weil-Deligne representation. Then
$$\det\left(-\Fr X\vL {r^I/r_N^I}\right) = \frac{L(r^{\vee}, \frac{1}{qX})}{L(r,X)}\frac{L((r,N),X)}{L((r^{\vee},N^{\vee}),\frac{1}{qX})}.$$
\end{lemma}
\begin{proof}
Since $L(r,X)=\det\left(I-r(\Fr)X \vL {r^I/r_N^I}\right)^{-1}\det\left(I-r(\Fr)X\vL r_N^I\right)^{-1}$, and similarly for $L(r^{\vee}, \frac{1}{qX})$, the Lemma can be reduced to proving that $$\det\left(I-r^{\vee}(\Fr)(qX)^{-1}\vL {(r^{\vee})^I/(r^{\vee}_{N^{\vee}})^I}\right) = \det\left(I-r(\Fr)^{-1}X^{-1}\vL {r^I/r_N^I}\right).$$ Since both sides are multiplicative in direct sums, it suffices to prove the equality for indecomposable Weil-Deligne representations, or in other words representations of the form $r'\otimes \Sp(n)$ for $r'$ an irreducible smooth representation of $W_F$ (see \cite[4.1.4]{tate_ntb} for the definition). In this case, the equality can be checked by direct computation.
\end{proof}

\begin{cor}
\label{gammasemisimplecharzero}
\begin{enumerate}
\item Given a Weil-Deligne representation $(r,N)$, $$\gamma((r,N),X,\psi)=\gamma(r,X,\psi).$$
\item Given an exact sequence $0\rightarrow V'\rightarrow V\rightarrow V''\rightarrow 0$ of Weil-Deligne representations (equivalently, $\ell$-adic representations), then $$\gamma(V,X,\psi)=\gamma(V',X,\psi)\gamma(V'',X,\psi).$$
\item Let $\rho$ be an $\ell$-adic representation of $W_F$, let $\rho^{ss}$ be its semisimplification, let $(r,N)$ be its associated Weil-Deligne representation, and let $(r^{ss},N)$ be the Frobenius-semisimplification of $(r,N)$. Then
$$\gamma(\rho,X,\psi)=\gamma(\rho^{ss},X,\psi) =\gamma(r^{ss},X,\psi).$$
\end{enumerate}
\end{cor}
\begin{proof}
Each of the constants $L$, $\epsilon$, and $\gamma$ remains the same if we replace $(r,N)$ by its Frobenius semisimplification, so Lemma \ref{extendingtoweildeligne} proves the first claim. The second claim follows immediately from the first, and the third follows from the second.
\end{proof}

Let $\rho$ be an $\ell$-adic representation over $\kappa$. From $\epsilon(\rho,\psi)$ we define a monomial $\epsilon(\rho,X,\psi)\in \kappa[X,X^{-1}]$ by twisting by the $\kappa[X,X^{-1}]^{\times}$-valued unramified character $\Fr\mapsto X$, and similarly for $\epsilon_0$:

\begin{align}
\label{epsilonwithX}
\epsilon(\rho,X,\psi)&:=\epsilon(\rho,\psi)X^{a(\rho) + (\dim\rho)n(\psi)}\\
\label{modifiedepsilonwithX}
\epsilon_0(\rho,X,\psi)&:=\epsilon_0(\rho,\psi)X^{\Sw(\rho)+(\dim\rho)(n(\psi)+1)}
\end{align}

It follows that $\epsilon(\rho,X,\psi)=\epsilon_0(\rho,X,\psi)\det\left(-\Fr X\vL {\rho^I}\right)^{-1}$ and hence 
\begin{equation}
\label{gammaintermsofmodifiedepsilon}
\gamma(\rho,X,\psi) = \epsilon_0(\rho,X,\psi)\det\left(-\Fr X\vL {\rho^I}\right)^{-1}\frac{L(\rho^{\vee},\frac{1}{qX})}{L(\rho,X)}.
\end{equation}

\section{Interpolating $\epsilon_0(\rho,X,\psi)$}
\label{epsilon0}
In this section we extend a theorem of Yasuda to the setting of Section \ref{arbitrary}.
\begin{thm}[\cite{yasuda}]
\label{yasuda}
For $R$ a complete Noetherian local ring with residue field $k$, let $\rho:W_F\to GL_n(R)$ be a continuous representation, and let $\psi:F\rightarrow R^{\times}$ be a smooth character. Then there exists an element $\epsilon_0(\rho,\psi)$ in $R^{\times}$, depending only on the isomorphism class of $\rho$, satisfying:
\begin{enumerate}
\item If $f:R\to R'$ is a local ring homomorphism, then $$f(\epsilon_0(\rho,\psi))=\epsilon_0(\rho\otimes_RR',\psi).$$
\item Given an exact sequence $0\to \rho'\to\rho\to\rho''\to 0$, we have $\epsilon_0(\rho,\psi)=\epsilon_0(\rho',\psi)\epsilon_0(\rho'',\psi)$.
\item If $R$ is a field, $\epsilon_0(\rho,\psi)$ equals the modified epsilon factor of Theorem \ref{modifiedepsilonfactor}.
\end{enumerate}
Moreover, the map $(R,\rho,\psi) \mapsto \epsilon_0(\rho,\psi)$ is uniquely determined by the above properties.
\end{thm}
In this setting we define the monomial $\epsilon_0(\rho,X,\psi)$ as in Equation \ref{modifiedepsilonwithX}.

\begin{prop}
\label{epsilon_0reducedelltorsionfree}
Let $R$ be a Noetherian $W(k)$-algebra and let $\rho:W_F\to GL_n(R)$ be an $\ell$-adically continuous representation of $W_F$. Then there exists an element $\epsilon_0(\rho,\psi)\in R^{\times}$, depending only on the isomorphism class of $\rho$ such that:
\begin{enumerate}
\item If $f: R \rightarrow R'$ is a map of Noetherian $W(k)$-algebras, then
$$f(\epsilon_0(\rho,\psi)) = \epsilon_0(\rho \otimes_R R',\psi).$$
\item Given an exact sequence $0 \to \rho' \to \rho \to \rho'' \to 0$, we have
$\epsilon_0(\rho,\psi) = \epsilon_0(\rho',\psi)\epsilon_0(\rho'',\psi).$
\item If $R$ is a field, then $\epsilon_0(\rho,\psi)$ equals the modified epsilon factor of Theorem
\ref{modifiedepsilonfactor}.
\end{enumerate}
Moreover, the map $(R,\rho,\psi) \mapsto \epsilon_0(\rho,\psi)$ is uniquely determined by the above properties.
\end{prop}
\begin{proof}
We first prove this in the case $R = R^{\nu}, \rho = \rho_{\nu}$ for some $\nu$ as in Section \ref{arbitrary}.  In particular $R$ is reduced and $\ell$-torsion free.

Let $R'$ be the total quotient ring of $R$; then $R'$ is a product of fields of characteristic zero, so we may construct an element $\epsilon_0(\rho\otimes_RR',\psi)$. On the other hand if $x$ is any point of $\Spec(R)$ and $\hat{R}_x$ denotes the completion of the local ring $R_x$ at $x$ then we have $\epsilon_0(\rho\otimes_R\hat{R}_x,\psi)$ from Theorem \ref{yasuda}.

If $\hat{R}_x'$ denotes the total quotient ring of $\hat{R}_x$, we have natural maps $R'\to \hat{R}'_x$ and $\hat{R_x}\to \hat{R}'_x$. Since $\epsilon_0(\rho\otimes_RR',\psi)$ is compatible with base change component-wise, the images of $\epsilon_0(\rho\otimes_RR',\psi)$ and $\epsilon_0(\rho\otimes_R\hat{R}_x,\psi)$ in $\hat{R}_x'$ coincide. Therefore we would like to show that if an element $\epsilon$ of $R'$ has the property that for all $x$ there exists an element $\epsilon_x$ of $\hat{R}_x$ whose image in $\hat{R}_x'$ coincides with the image of $\epsilon$ in $\hat{R}_x'$, then $\epsilon$ lies in $R$.

We first show that $\epsilon_x$ lies in the localization $R_x$. Since there exists some nonzerodivisor $b$ of $R$ such that $b\epsilon$ lies in $R$, the product $b\epsilon_x$ must lie in the subset $R_x$ of $\hat{R}_x$. We must show that $b\epsilon_x$ lies in $bR_x$. Completing the exact sequence $$0\to R_x\to R_x\to R_x/bR_x\to 0$$ shows that the completion of the local ring $R_x/bR_x$ at $x$ is $\hat{R}_x/b\hat{R}_x$, and there is an injection $R_x/bR_x\to \hat{R}_x/b\hat{R}_x$. This shows that an element of $b\hat{R}_x$ that is in $R_x$ lies in $bR_x$. Therefore $b\epsilon_x$ lies in $bR_x$ and, since $b$ is a nonzerodivisor, it must be that $\epsilon_x$ lies in $R_x$.

Thus we can view $\epsilon$ and $\epsilon_x$ as functions defined on open subsets of $\Spec(R)$. Let $U$ be an open subset on which $\epsilon$ is defined. Since $\epsilon$ maps to each $\epsilon_x$ in the total quotient ring $R_x'$, there is a neighborhood $U_x$ of $x$ such that $\epsilon_x$ is defined on $U_x$ and such that the restriction of $\epsilon$ to $U\cap U_x$ agrees with $\epsilon_x$. By gluing, we see that $\epsilon$ extends to $\Spec(R)$ and therefore defines an element of $R$. If $\epsilon$ and each $\epsilon_x$ were units in $R'$ and $\hat{R}_x$, the same argument shows that $\epsilon^{-1}\in R$.

Now assume $R$ and $\rho$ are arbitrary satisfying the hypotheses of the proposition.  Then, since $R$ is Noetherian, $\Spec(R)$ has finitely many connected components, and we easily reduce to the case where $\Spec(R)$ is connected.  Then $\rho$ has a type $\nu$.  There is thus an affine open cover $U_j = \Spec R_j$ of $\Spec R$, and for each $j$, a map $f_j$ from $R^{\nu}$ to $R_j$ such that the representation $\rho_j$ given by $\rho \otimes_R R_j$ arises by base change from $R^{\nu}$.  We then set:
$$\epsilon_0(\rho_j,\psi) = f_j(\epsilon_0(\rho_{\nu},\psi)).$$

We must show that for any $j,j'$ the restrictions to $U_j \cap U_{j'}$ of $\epsilon_0(\rho_j,\psi)$ and $\epsilon_0(\rho_{j'},\psi)$ agree.  Certainly the restrictions of $\rho_j$ and $\rho_{j'}$ to $U_j \cap U_{j'}$
agree, but the maps $f_j$ and $f_{j'}$ are induced by choices of pseudo-framings of $\rho_j$ and $\rho_{j'}$
and neither these choices nor the corresponding maps need agree on the overlap.  However, if we set $U_j \cap U_{j'} = \Spec R_{jj'}$, there is an element $g$ of $G^{\nu}(R_{jj'})$ that carries the pseudo-framing of $\rho_j$ to that of
$\rho_{j'}$.  Since $G^{\nu}$ is a product of groups of the form $GL_{n_{\tau}}$ we may regard $g$ as a collection of matrices with entries in $R_{jj'}$.

If the entries of $g$ were in the image of the map $f_j: R^{\nu} \rightarrow R_{jj'}$, we could lift $g$ to an element
$\tilde g$ of $G^{\nu}(R^{\nu})$, and we would have $f_{j'} = f_j \circ \tilde g$.  Since $\rho_{\nu}$ is invariant under $\tilde g$, so is $\epsilon_0(\rho_{\nu},\psi)$, and thus 
$$f_{j'}(\epsilon_0(\rho_{\nu}, \psi)) = f_j(g(\epsilon_0(\rho_{\nu},\psi))) = f_j(\epsilon_0(\rho_{\nu},\psi))$$
as desired.

Of course, in general there is no reason to expect that we can lift $g$ to an element of $G^{\nu}(R^{\nu})$.  Let
${\tilde R}^{\nu}$ be the ring $R^{\nu}[X_1,\dots,X_r]$, where there is one indeterminate $X_r$ for each entry of
$g$.  We have a map ${\tilde f}_j: {\tilde R}^{\nu} \rightarrow R_{jj'}$ that agrees with $f_j$ on $R^{\nu}$ and sends each $X_i$ to the corresponding entry of $g$.  It is then clear that we can lift $g$ to an element ${\tilde g}$ of ${\tilde R}^{\nu}$.

Let ${\tilde \rho}^{\nu}$ be the base change of $\rho^{\nu}$ to ${\tilde R}^{\nu}$ under the natural inclusion $h_j$ of $R^{\nu}$ into ${\tilde R}^{\nu}$.  Acting by ${\tilde g}$ on the pseudo-framing of ${\tilde \rho}^{\nu}$ gives a new map $h_{j'}: R^{\nu} \rightarrow {\tilde R}^{\nu}$.  We have $f_{j'} = {\tilde f}_j \circ h_{j'}$, and $f_j = {\tilde f}_j \circ h_j$.  It thus suffices to show that $h_j(\epsilon_0(\rho_{\nu},\psi)) = h_{j'}(\epsilon_0(\rho_{\nu},\psi))$.

Let $x$ be a map from ${\tilde R}^{\nu}$ to an algebraically closed field $K$ of characteristic zero.  Then the base changes of $\rho^{\nu}$ along $x \circ h_j$ and $x \circ h_{j'}$ are isomorphic as representations; their pseudo-framings differ by the element $x(\tilde g)$ of $G^{\nu}(K)$.  Thus we have:
$$x(h_j(\epsilon_0(\rho_{\nu},\psi))) = x(h_{j'}(\epsilon_0(\rho_{\nu},\psi))).$$
Since ${\tilde R}^{\nu}$ is reduced and $\ell$-torsion free we have
$$h_j(\epsilon_0(\rho_{\nu},\psi)) = h_{j'}(\epsilon_0(\rho_{\nu},\psi))$$
as claimed.

Therefore the elements $f_j(\epsilon_0(\rho_{\nu},\psi))$ glue to form an element $\epsilon_0(\rho,\psi)$ of $R$.

It is clear that with this definition the association $(R,\rho,\psi) \mapsto \epsilon_0(\rho,\psi)$ has the claimed properties.  For uniqueness, note that since $R_{\nu}$ is reduced and $\ell$-torsion free, characteristic zero points of $\Spec(R_{\nu})$ are dense in $\Spec(R_{\nu})$.  Thus the listed properties uniquely determine $\epsilon_0(\rho_{\nu},\psi)$.  Since any $\rho$ arises Zariski locally by base change from this case, we have uniquely determined $\epsilon_0(\rho,\psi)$ for an arbitrary $\rho$ as well.
\end{proof}
Thus, for any Noetherian $W(k)$-algebra with connected spectrum (or for $R^{\nu}$) we can define the monomial $\epsilon_0(\rho,X,\psi)$ as in Equation \ref{modifiedepsilonwithX}, using the interpolated Swan conductor $\Sw_R$.

\begin{rmk} \rm The argument used to construct $\epsilon_0(\rho,\psi)$ is slightly ad-hoc.  The reader who is acquainted with stacks might prefer the following more abstract argument (which is essentially equivalent to the above): the element
$\epsilon_0(\rho^{\nu},\psi)$ is $G^{\nu}$-invariant (as can be checked on characteristic zero points), and so descends to a function on the quotient stack $(\Spec R^{\nu})/G^{\nu}$.  As explained earlier, this quotient stack is a moduli stack for $\ell$-adically continuous representations of $W_F$ of type $\nu$ on locally free $R$-modules.  Thus for an arbitrary $R$, an $\ell$-adically continuous $\rho$ gives rise to a unique map from $\Spec R$ to $(\Spec R^{\nu})/G^{\nu}$, and $\epsilon_0(\rho,\psi)$ is simply the pullback of $\epsilon_0(\rho^{\nu},\psi)$ along this map.
\end{rmk}

\section{Interpolating the ratio of $L$-functions}
\label{provingrestofgamma}
An easy calculation shows that 
\begin{align*}
\det\left(-\Fr X\vL {\rho^I}\right)^{-1}\frac{L(\rho^{\vee},\frac{1}{qX})}{L(\rho,X)} &=\frac{q^{\dim\rho^I}\det\left(I -\Fr X\vL{\rho^I}\right)}{\det\left(-\Fr qX\vL{\rho^I}\right)\det\left(I-\Fr^{-1}(qX)^{-1}\vL{\rho^I}\right)}\\
&=\frac{q^{\dim\rho^I}L(\rho,qX)}{L(\rho,X)}
\end{align*}

The ratio $\frac{q^{\dim\rho^I}L(\rho,qX)}{L(\rho,X)}$ divides $\gamma(\rho,X,\psi)$, and is defined in terms of the space of inertial invariants, which does not interpolate. Our next result, Theorem \ref{restofgamma}, says that when $\rho$ is tamely ramified semisimple this ratio admits an explicit description in terms of characteristic polynomials (denoted ``$\Char$'' below) of operators generated by $\Fr$ and a generator $\sigma$ of tame inertia, without any mention of inertial invariants. Theorem \ref{restofgamma} will allow us to deduce in Section \ref{gammaforrings} that $\gamma(\rho,X,\psi)$ interpolates.

\begin{thm}
\label{restofgamma}
Let $\kappa$ be a field of characteristic zero, and let $\rho:W_F\to GL_n(\kappa)$ be a semisimple smooth representation. Let $\sigma$ be the image in $\rho$ of a topological generator of tame inertia. If $\rho$ is tamely ramified, then
\begin{align}
\label{theeqn1}
\det(1+\sigma+...+\sigma^{q-1})&=q^{\dim\rho^I}\\
\label{theeqn2}
\frac{\Char(\Fr)(X)}{\Char\big((1+\sigma+...+\sigma^{q-1})\Fr\big)(X)}&=\frac{L(\rho,qX)}{L(\rho,X)}
\end{align}
\end{thm}
Before embarking on the proof of Theorem \ref{restofgamma} we will require some machinery. 
Given an extension $E/F$ of fields, let $N_{E/F}$ denote the norm map $E^{\times}\to F^{\times}$ and given a character $\xi$ of $E^{\times}$ we treat $\xi$ as a character of $W_E$ via local class field theory, and $\Ind_{E/F}\xi$ will denote the representation of $W_F$ induced from the subgroup $W_E$.
\begin{defn}
Let $E/F$ be a finite tamely ramified extension and let $\xi$ be a character of $E^{\times}$. The pair $(E/F,\xi)$ is called \emph{admissible} if it satisfies the following two conditions as $K$ ranges over intermediate fields $F\subset K\subset E$:
\begin{enumerate}
\item If $\xi$ factors through the relative norm $N_{E/K}$, then $K=E$.
\item If $\xi|_{U_E^1}$ factors through $N_{E/K}$ then $E/K$ is unramified.
\end{enumerate}
\end{defn}
Suppose $\rho$ is an irreducible smooth tamely ramified representation of $W_F$ of dimension $d>1$. By \cite[A.3 Theorem]{bh_essentiallytame}, $\rho$ must have the form $\Ind_{E/F}\xi$ where $(E/F,\xi)$ form an admissible pair with $[E:F]=d$ (we implicitly identify $W_E^{ab}$ with $E^{\times}$ via local class field theory). Moreover, for $g\in \Gal(E/F)$, $g\neq 1$, the characters $\xi,\xi^{g},\xi^{g^2},\dots,\xi^{g^{d-1}}$, of $W_E^{ab}$ are all distinct.

\begin{lemma}
\label{irreducibleinductionisunramified}
Suppose $\rho$ is irreducible and tamely ramified of dimension $d>1$ over a field of characteristic zero, thus having the form $\Ind_{E/F}\xi$ for an admissible pair $(E/F,\xi)$. Suppose $\sigma$ is a topological generator of the tame inertia quotient $I_F/P_F$.
\begin{enumerate}
\item $E/F$ is unramified,
\item $\rho(\sigma)=\left(\begin{smallmatrix}\xi(\sigma)\\
&\ddots&\\
&& \xi^{q^{d-1}}(\sigma)\end{smallmatrix}\right)$,
\item $\xi$ is tamely ramified, $\xi(\sigma)$ is a $q^d-1$'st root of unity, and $d$ is minimal among positive integers $m$ satisfying $\xi(\sigma)^{q^m-1}=1$.
\item $\rho(\Fr)$ is the product of a diagonal matrix with a permutation matrix corresponding to a $d$-cycle in $S_d$.
\end{enumerate}
\end{lemma}
\begin{proof}
We use Mackey's restriction formula $\rho|_{I_F} = \bigoplus_{y\in W_E\backslash W_F / I_F}\Ind_{I_F\cap W_E^y}^{I_F}(\xi^y)$. If $\xi$ were nontrivial on the wild inertia subgroup $P_E$, then $\rho$ would be nontrivial on $P_F$. Therefore, the character $\xi|_{U_E^1}$ is trivial. Since $(E/F,\xi)$ is an admissible pair, $E/F$ must be unramified. In this case, $I_E=I_F$ and there are $d$ double cosets in $W_E\backslash W_F/I_F$ with representatives $1,\Fr,\dots,\Fr^{d-1}$. Since $\rho$ has dimension $d$, each component is one-dimensional, and therefore given by $\xi^{\Fr^i}$ for powers $i$. Replacing $\rho(x)$ with $\rho(\Fr x \Fr^{-1})$ has the effect of cyclically permuting the summands. Since $\sigma$ satisfies $\Fr\sigma\Fr^{-1}=\sigma^q$, we have $\xi^{\Fr}=\xi^q$. Since  $\xi,\xi^{\Fr},\xi^{\Fr^2},\dots,\xi^{\Fr^{d-1}}$ are all distinct, $\xi(\sigma)$ is a $q^d-1$'st root of unity, and $d$ is minimal with this property.
\end{proof}

The second lemma is a simple fact from linear algebra:

\begin{lemma}
\label{permutationdiagonal}
Suppose $D$ is a $d\times d$ diagonal matrix and $P$ is a $d\times d$ permutation matrix corresponding to a $d$-cycle. Then the characteristic polynomial of $DP$ has the form $X^d + (-1)^d\det(DP)$.
\end{lemma}
\begin{proof}
If $v_1$ is the first basis vector, and $T=DP$, then $\{v_1, Tv_1,\dots,T^{d-1}v_1\}$ is a new basis. The matrix for $T$ in the new basis is $\left(\begin{smallmatrix}
&\det(D)\\
I_{d-1}&\end{smallmatrix}\right)$.
\end{proof}

\begin{proof}[Proof of Thm \ref{restofgamma}]
Since both sides of Equations (\ref{theeqn1}) and (\ref{theeqn2}) are multiplicative in direct sums, we can reduce from the semisimple case to the irreducible case.

Suppose $\rho$ is irreducible, so that according to Lemma \ref{irreducibleinductionisunramified}, there are three possible cases: $\rho$ is an unramified character, $\rho$ is a tamely ramified character, or $\dim\rho>1$ and $\rho=\Ind_{E/F}\xi$ as in Lemma \ref{irreducibleinductionisunramified}.
\begin{enumerate}
\item If $\rho$ is an unramified character, $\sigma = 1$. Therefore, $\det(1+\dots+\sigma^{q-1})=q=q^{\dim\rho^I}$, and $(1+\dots+\sigma^{q-1})\Fr=q\Fr$.
\item If $\rho$ is a tamely ramified character, then $\sigma$ is a primitive $q-1$'st root of unity, so $1+\dots+\sigma^{q-1}=1=q^{\dim\rho^I}$ and $(1+\dots+\sigma^{q-1})\Fr=\Fr$.
\item If $\rho$ is irreducible of dimension $d>1$, it is of the form $\Ind_{E/F}\xi$ where $\xi$ is a tamely ramified character of $W_E^{ab}$. Therefore $1+\dots+\sigma^{q-1}$ is equivalent to $$\left(\begin{smallmatrix}1+\xi(\sigma) \dots+\xi(\sigma)^{q-1}\\
&1+\xi(\sigma)^q+\dots+\xi(\sigma)^{q(q-1)}\\
&&\ddots&\\
&&& 1+\xi(\sigma)^{q^{d-1}}+\dots +\xi(\sigma)^{(q-1)q^{d-1}}\end{smallmatrix}\right).$$ Since $E/F$ is unramified of degree $d$, $\xi(\sigma)$ must be a nontrivial $q^d-1$'st root of unity. An easy calculation shows that $$\det(1+\dots+\sigma^{q-1})=1+\dots+\xi(\sigma)^{q^d-1} = 1=q^{\dim\rho^I}.$$ By Lemma \ref{irreducibleinductionisunramified}, $\Fr=DP$ is a diagonal matrix $D$ times a permutation matrix $P$. Since $1+\dots+\sigma^{q-1}$ is a diagonal matrix, so is the product $(1+\dots+\sigma^{q-1})D$. Therefore Lemma \ref{permutationdiagonal}, tells us that
\begin{align*}
\Char\big((1+\dots+\sigma^{q-1})\Fr\big)(X)&=X^d + (-1)^d\det\big((1+\dots+\sigma^{q-1})DP\big)\\
& = X^d + (-1)^d\det(\Fr)\\
&=\Char(\Fr)(X).
\end{align*}
\end{enumerate}
\end{proof}

\section{Proof of main theorem}
\label{gammaforrings}

In this section we prove Theorem \ref{mainthm}. The two key results used to establish Theorem \ref{mainthm} in the complete local setting are Theorem \ref{restofgamma} and Proposition \ref{epsilon_0reducedelltorsionfree}. 

\begin{lemma}
\label{frameinvariantsemisimple}
Let $\rho$ be an $\ell$-adic representation over $\kappa$. Let $w$ be an element of the group ring $\kappa[W_F]$, and consider $\rho$ as a homomorphism $\kappa[W_F]\to M_n(\kappa)$. Then $\det(\rho(w))=\det(\rho^{ss}(w))$ and $\Char(\rho(w))=\Char(\rho^{ss}(w))$.
\end{lemma}
\begin{proof}
Both $\det$ and $\Char$ are multiplicative in exact sequences.
\end{proof}

\begin{proof}[Proof of Theorem \ref{mainthm}] 
We say that $\rho$ is totally wildly ramified if $\rho^{P_F}=\{0\}$, where $P_F$ is the wild inertia subgroup of $I_F$. In the totally wildly ramified case, $\gamma(\rho,X,\psi)=\epsilon_0(\rho,X,\psi)$, so the result follows immediately from \cite[Thm 5.3]{yasuda} and Proposition \ref{elladicswan}.

In general, by Corollary \ref{cor:sum}, $\rho$ may be decomposed as a direct sum $\rho^0\oplus\rho^{>0}$ where $\rho^0$ is tamely ramified and $\rho^{>0}$ is totally wildly ramified. Therefore, if we had a suitable interpolation of $\gamma(\rho^0,X,\psi)$, we could define $$\gamma(\rho,X,\psi):=\gamma(\rho^0,X,\psi)\gamma(\rho^{>0},X,\psi) =\gamma(\rho^0,X,\psi)\epsilon_0(\rho^{>0},X,\psi),$$ and it would satisfy the required properties. Therefore we assume $\rho$ is tamely ramified.

We now prove the theorem when $(R,\rho)$ is of the form $(R^{\nu},\rho_{\nu})$.  In particular $R$ is reduced and $\ell$-torsion free.

For any characteristic zero point $x\in \Spec(R)$, $\det(\rho_x(1+\cdots+\sigma_x^{q-1}))=\det(\rho_x^{ss}(1+\cdots+\sigma_x^{q-1}))$ is a power of $q$ (see the proof of Theorem \ref{restofgamma}) and therefore defines a power of $q$ in $R/x$. Since any maximal ideal of $R$ contains such an $x$, $\det(\rho_x(1+\cdots+\sigma_x^{q-1}))$ is invertible modulo every maximal ideal, so is a unit in $R$. Thus the element $$\gamma_R(\rho,X,\psi):=\epsilon_0(\rho,\psi)X^{\Sw_R(\rho)+(\dim\rho)(n(\psi)+1)}\frac{\det(1+\sigma+...+\sigma^{q-1})\Char(\Fr)(X)}{\Char\big((1+\sigma+...+\sigma^{q-1})\Fr\big)(X)}$$ lives in $S^{-1}R[X,X^{-1}]$. Suppose that $x:R\to \kappa$ is a map to a field $\kappa$ of characteristic zero which is the fraction field of a complete local domain with residue field $k$. Compatibility with specialization means $\gamma_R(\rho,X,\psi)(x)$ coincides with the specialization
$$\epsilon_0(\rho_x,\psi_x)X^{\Sw(\rho_x)+(\dim\rho_x)(n(\psi_x)+1)}\frac{\det(1+\sigma_x+...+\sigma_x^{q-1})\Char(\Fr_x)(X)}{\Char\big((1+\sigma_x+...+\sigma_x^{q-1})\Fr_x\big)(X)}$$ at $x$. But $\epsilon_0(\rho_x,\psi_x)=\epsilon_0(\rho_x^{ss},\psi_x)$, and $$X^{\Sw(\rho_x)+(\dim\rho_x)(n(\psi_x)+1)}=X^{\Sw(\rho_x^{ss})+(\dim\rho_x^{ss})(n(\psi_x)+1)}.$$ Since $\rho_x^{ss}$ is smooth, Theorem \ref{restofgamma}, together with Lemma \ref{frameinvariantsemisimple}, tells us
$$\frac{\det(1+\sigma_x+...+\sigma_x^{q-1})\Char(\Fr_x)(X)}{\Char\big((1+\sigma_x+...+\sigma_x^{q-1})\Fr_x\big)(X)} = \frac{q^{\dim(\rho_x^{ss})^I}L(\rho_x^{ss},qX)}{L(\rho_x^{ss},X)}.$$ Therefore $\gamma_R(\rho,X,\psi)(x)=\gamma(\rho_x^{ss},X,\psi_x)$. But by Corollary \ref{gammasemisimplecharzero}, $\gamma(\rho_x^{ss},X,\psi_x)=\gamma(\rho_x,X,\psi_x)$.

If $(R,\rho)$ is arbitrary, then as in the proof of Proposition \ref{epsilon_0reducedelltorsionfree}, $\rho$ arises, Zariski locally on $\Spec R$, by base change from some product of $(R^{\nu_i},\rho_{\nu_i})$, and we define $\gamma_R(\rho,X,\psi)$ to be the series obtained by gluing the appropriate base changes of $\prod_i \gamma_{R^{\nu_i}}(\rho_{\nu_i},X,\psi)$.  The proof of uniqueness proceeds just as in the proof of Proposition \ref{epsilon_0reducedelltorsionfree}.
\end{proof}

\bibliography{mybibliography}{}
\bibliographystyle{alpha}
\end{document}